\documentclass[a4paper,11pt]{amsart}
\usepackage[utf8]{inputenc}
\usepackage[T1]{fontenc}
\usepackage{graphicx,xcolor,url}
\usepackage{amsmath,amssymb,amsthm,mathtools}
\usepackage[hidelinks]{hyperref}
\usepackage{aliascnt}
\usepackage[nameinlink,noabbrev]{cleveref}
\usepackage[a4paper,margin=2.6cm]{geometry}

\newaliascnt{cor}{thm}
\newtheorem{cor}[cor]{Corollary}
\aliascntresetthe{cor}
\newaliascnt{obs}{thm}
\newtheorem{obs}[obs]{Observation}
\aliascntresetthe{obs}
\newaliascnt{prop}{thm}

\aliascntresetthe{prop}
\newaliascnt{lem}{thm}
\newtheorem{lem}[lem]{Lemma}
\aliascntresetthe{lem}
\newaliascnt{rem}{thm}

\aliascntresetthe{rem}
\newaliascnt{ques}{thm}
\newtheorem{ques}[ques]{Question}
\aliascntresetthe{ques}
\newaliascnt{conj}{thm}

\aliascntresetthe{conj}
\theoremstyle{definition}
\newaliascnt{defi}{thm}

\aliascntresetthe{defi}
\Crefname{thm}{Theorem}{Theorems}
\Crefname{lem}{Lemma}{Lemmas}
\Crefname{prop}{Proposition}{Propositions}
\Crefname{cor}{Corollary}{Corollaries}
\Crefname{obs}{Observation}{Observations}
\Crefname{rem}{Remark}{Remarks}
\Crefname{defi}{Definition}{Definitions}
\Crefname{ques}{Question}{Questions}
\Crefname{conj}{Conjecture}{Conjectures}
\usepackage{thmtools}

\newcommand{\Span}{\operatorname{span}}

\newcommand{\girth}{\operatorname{girth}}
\newcommand{\F}{\mathbb F}

\title{Rings are $\chi$-bounded}

\author[S. Asensio]{Sara Asensio}
 \address{Instituto de Investigaci\'on en Matem\'aticas (IMUVa), Universidad de Valladolid, Valladolid, Spain}
 \email{sara.asensio@uva.es}

\author[I. García-Marco]{Ignacio García-Marco}
 \address{Instituto de Matem\'aticas y Aplicaciones (IMAULL), Secci\'on de Matem\'aticas, Facultad de
Ciencias, Universidad de La Laguna, 38200, La Laguna, Spain}
 \email{iggarcia@ull.edu.es}

\author[K. Knauer]{Kolja Knauer}
\address{Departament de Matemàtiques i Informàtica,
 Universitat de Barcelona, Barcelona, Spain;
 Centre de Recerca Matemàtica (CRM),
 Campus de Bellaterra, Edifici C,
 08193 Bellaterra, Barcelona, Spain.}
 \email{kolja.knauer@ub.edu}

\begin{document}
\maketitle

\begin{abstract}

We prove that there is a function $f:\mathbb N\to\mathbb N$ such that
$\chi(\Gamma(R))\le f(\omega(\Gamma(R)))$ 
for the zero-divisor graph of any ring $R$, if the clique number is
finite. 
On the one hand, this consolidates a disproved conjecture of Beck from 1988, claiming $\chi(\Gamma(R))=\omega(\Gamma(R))$ for unital commutative rings. While previous counterexamples satisfy  $\chi(\Gamma(R))\le \omega(\Gamma(R))+2$, we obtain the lower bound
$f\geq k^{\Omega(\log k)}$ among finite commutative rings. Finally, we show that neither zero-divisor graphs of finite commutative semirings nor those of finite commutative nonassociative rings are $\chi$-bounded.
\end{abstract}

\section{Introduction}\label{sec:introduction}

Unless specified otherwise, neither commutativity nor a neutral element are assumed for the multiplication in a \emph{ring} $R$ in this paper.
 
A nonzero element $x\in R$ is a \emph{one-sided zero-divisor} if $xy=0$ or $yx=0$
for some nonzero $y\in R$.  The \emph{zero-divisor graph} $\Gamma(R)$ has
the one-sided zero-divisors as vertices, with distinct vertices
$x,y$ adjacent when
\[
 xy=0\quad\text{or}\quad yx=0.
\]
For commutative rings this is the usual zero-divisor graph~\cite{AL99}. For noncommutative rings, a directed version with an arc $x\to y$ whenever $xy=0$ has been studied, see~\cite{Redmond02,Wu05,AM06}. In this case our graph is
the underlying simple graph obtained by forgetting the directions.

We write $\chi(\Gamma(R))$ and $\omega(\Gamma(R))$ for its chromatic and
clique numbers.  
The coloring problem originates in Beck's paper on commutative unital rings $R$
\cite{Beck88}\footnote{Beck considered the graph on all elements of a unital
commutative ring, with distinct elements adjacent when their product is zero.
After deleting the nonzero regular elements, which are adjacent only to $0$,
this differs from the usual zero-divisor graph only by adjoining $0$ as a
universal vertex and $\chi$-boundedness is not affected}, where he conjectured that $\chi(\Gamma(R))=\omega(\Gamma(R))$ in that setting. Beck's conjecture, or even perfectness of $\Gamma(R)$, is known for several important classes of rings, see~\cite{AAS11,PWJ17}. However, Anderson and Naseer disproved this conjecture \cite{AN93}, and
further counterexamples and families were developed in
\cite{AAS11,BDS98,DS98,Vietri13,Vietri15,TLW20,Kav22}. Our main result establishes $\chi$-boundedness for arbitrary rings. 

\begin{restatable}{thm}{ringchibounded}\label{thm:ring-chi-bounded}
There exists a function $f:\mathbb N\to\mathbb N$ such that $\chi(\Gamma(R))\le f\bigl(\omega(\Gamma(R))\bigr)$, for every
 ring $R$ with $\omega(\Gamma(R))<\infty$.
\end{restatable}
The finite-clique hypothesis is
essential: Hala\v{s} and P\'ocs showed that no analogous statement holds for
infinite cardinals \cite{HP24}. 

We prove \Cref{thm:ring-chi-bounded} in \Cref{sec:uniform-bound}.  The proof
starts from a finite-ideal Ramsey lemma in which the direction of a zero
product is recorded as part of the edge color.  The prime radical $P$ is then
controlled using Bell's theorem that every infinite nil ring contains an
infinite zero subring \cite{Bell01}.  Bell's structure theorem for semiprime
rings with no infinite zero subrings reduces the quotient $R/P$ to a reduced ring
and finitely many full matrix rings over finite fields \cite{Bell01}; the
reduced part is handled using the Andrunakievi\v{c}--Rjabuhin theorem in
Klein's form \cite{Klein80} and the coloring theorem for meet-semilattices
\cite{NWD07}.  

We want to go further and measure the quality of the binding function $f$ in \Cref{thm:ring-chi-bounded}. 
Going along the proof yields a tower-type upper bound of $f$, see~\Cref{cor:asymptotic-upper}.

For lower bounds the classical counterexamples, Anderson--Naseer's ring
has $(\omega,\chi)=(4,5)$, and the direct-product
formulas of \cite{Kav22} give that $f(k)\ge k+1$ for every $k\ge4$.
Dumaldar--Sharma's $p=7$ example has $(\omega,\chi)=(49,51)$~\cite{DS98,TLW20}, and the same product formulas give
$f(k)\ge k+2$ for every $k\ge49$.  
Via symplectic graphs one may obtain 
$f(k)=\Omega(k^2/\log k)$, 
see~\cite{GodsilRoyle01,TW06}. However, using several alternating forms
simultaneously yields a much stronger bound.

\begin{restatable}{thm}{superpolylower}\label{thm:superpoly-lower}
There is an infinite family of finite commutative unital rings $R_t$, with $t\geq 2$, such that 
\[ 
\omega(\Gamma(R_t))\leq 2^{O(t)}\qquad \text{and} \qquad \chi(\Gamma(R_t))\geq 2^{\Omega(t^2)}.\]
\end{restatable}

In particular, the binding function $f(k)$ is in $2^{\Omega((\log k)^2)}$. Hence, our upper and lower bounds leave a very large gap. We believe that the following would be an interesting measure for currently available tools and understanding:
\begin{ques}
    What is the asymptotic behavior of the binding function $f$ in \Cref{thm:ring-chi-bounded}?
\end{ques}

Finally, let us exploit the impossibility of extension of~\Cref{thm:ring-chi-bounded} to other classes.
The notion of a zero-divisor graph has also been extended to commutative
semigroups~\cite{DMS02,AB17}. There $\chi$-boundedness
cannot hold, since every graph with a dominating vertex is realizable as a
zero-divisor graph \cite[Theorem~3]{DD05}.
Also, zero-divisor
graphs of \emph{semirings} (i.e., both operations are semigroups) have been studied, see ~ \cite{DolzanOblak12}. 

Recall that a graph $G$ is a \emph{retract} of a graph $H$ if there is a graph homomorphism $H\to G$ restricting to the identity on $G$. In this case $\omega(G)=\omega(H)$ and $\chi(G)=\chi(H)$. Thus, with the abundance of graph families of bounded clique number and arbitrarily large chromatic number (see, e.g., \cite{ScottSeymour20,Reiher26}) we can construct many non-$\chi$-bounded families with the following two theorems:
 
\begin{restatable}{thm}{semiring}\label{thm:semiring-retract}
For every finite graph $G$ with a dominating vertex, there exists a finite
commutative unital semiring $R$ such that $G$ is a
retract of $\Gamma(R)$.  In particular, $\omega(G) = \omega(\Gamma(R))$ and $\chi(G) = \chi(\Gamma(R))$ and zero-divisor graphs of finite commutative unital semirings are not $\chi$-bounded.
\end{restatable}

Finally, also associativity cannot be omitted from~\Cref{thm:ring-chi-bounded}

\begin{restatable}{thm}
{free}\label{prop:free-realization}
For every finite graph $G$ without isolated vertices, there is a commutative nonassociative unital ring $R$ such that $G$ is a {retract} of $\Gamma(R)$.   In particular, $\omega(G) = \omega(\Gamma(R))$ and $\chi(G) = \chi(\Gamma(R))$ and zero-divisor graphs of commutative unital nonassociative rings are not $\chi$-bounded.
\end{restatable}

We do not know whether in \Cref{prop:free-realization} the ring $R$ can be assumed to be finite and we believe that this actually is an interesting question. 

\begin{ques}
Is every finite graph without isolated vertices a retract of the zero-divisor graph of a finite commutative unital nonassociative ring?
\end{ques}

We can however prove probabilistically, that zero-divisor graphs of finite commutative nonassociative rings are not $\chi$-bounded.

\begin{restatable}{thm}{nonassocunbounded}\label{thm:nonassoc-unbounded}
For every $k$ there is a finite commutative nonassociative unital ring $R$ such that $
 \omega(\Gamma(R))\le16$ and $\chi(\Gamma(R))\ge k$.
\end{restatable}

\section{A uniform bound for rings}
\label{sec:uniform-bound}

In this section we prove \Cref{thm:ring-chi-bounded}. Before that, we first isolate the additive
observation which will be used in the Ramsey argument.

\begin{lem}\label{lem:rainbow-differences}
Let $A$ be an abelian group, let $s\ge1$, and let
$X_1,\ldots,X_s\subseteq A$ satisfy $|X_i|\geq 2i$ for every $i$. Then one can
choose distinct $x_i,y_i\in X_i$, $i\in[s]$, such that the differences
$d_i=x_i-y_i$ are pairwise distinct up to sign; that is,
$d_i\notin\{d_j,-d_j\}$ whenever $i\ne j$.

\end{lem}

\begin{proof}
Choose the pairs successively. Suppose that we have already chosen the pairs corresponding to $X_1,\ldots,X_j$, where $j<s$, and consider $X_{j+1}$. Let us construct the graph with vertex set $X_{j+1}$ in which two vertices $x$ and $y$ are adjacent if $x-y\in\{\pm d_1,\dots,\pm d_j\}$, which implies that every vertex has degree at most $2j< |X_{j+1}|-1$. Hence for every $x\in X_{j+1}$ there is $y\in X_{j+1}$ such that $x-y\notin\{\pm d_1,\dots,\pm d_j\}$, and this provides a pair in $X_{j+1}$ with a new difference class up to sign. Continuing greedily proves the claim. 
\end{proof}

For two positive integers $q,s$, let $\mathcal R_q(s)$ denote the
\emph{multicolor Ramsey number}, i.e., the least integer $N$ such that every
$q$-edge-coloring of $K_N$ contains a monochromatic $K_s$. 

Given a ring $S$, a \textit{left} (resp. \textit{right}) \textit{ideal} of $S$ is an additive subgroup of $S$ which is closed under left (resp. right) multiplication by elements of $S$, while an ideal satisfying the two conditions is said to be a \textit{two-sided ideal}. Notice that in the commutative setting every ideal is a two-sided ideal. If $I$ is a left (resp. right) ideal, then the quotient $S/I$ is a left (resp. right) $S$-module. In this work we will consider two-sided ideals $I$, which satisfy that $S/I$ is a ring.

\begin{lem}\label{lem:finite-ideal-quotient}
Let $R$ be a ring and let $P\subseteq R$ be a finite two-sided ideal of size 
$m$. If $\omega(\Gamma(R))\le k$, then
\[
    \omega(\Gamma(R/P))
    \le
    \mathcal R_{2m}\bigl((k+1)(k+2)\bigr)-1.
\]

\end{lem}

\begin{proof}
Put $N=\mathcal R_{2m}\bigl((k+1)(k+2)\bigr)$ and suppose, for a contradiction,
that $\Gamma(R/P)$ contains a clique
$\{\overline{x_1},\ldots,\overline{x_N}\}$. Choose representatives
$x_1,\ldots,x_N\in R$. For every $i<j$, since $\overline{x_i}$ and $\overline{x_j}$ are
adjacent in $\Gamma(R/P)$, at least one of $x_ix_j,x_jx_i$ belongs to $I$.
Color the edge $\{\overline{x_i},\overline{x_j}\}$ by $(\to,x_ix_j)$ if $x_ix_j\in P$, and otherwise
by $(\leftarrow,x_jx_i)$. There are at most $2m$ colors.

By the definition of $\mathcal R_{2m}\bigl((k+1)(k+2)\bigr)$, there is a set
of $(k+1)(k+2)$ indices on which all edges have the same color. Keep these
indices in their inherited order and partition them into $k+1$ consecutive
blocks $X_1,\dots,X_{k+1}$, with $|X_i|=2i$ for every $i\in\{1,\dots,k+1\}$. Applying
\Cref{lem:rainbow-differences} to the sets of representatives belonging to
these blocks, we may choose two representatives
$a_i,b_i$ from the $i$th block such that $d_i=a_i-b_i$, for $i\in[k+1]$, are pairwise distinct up
to sign. Since $a_i$ and $b_i$ represent distinct cosets modulo $P$, every
$d_i$ is nonzero.

Suppose first that the common Ramsey color is $(\to,c)$ for some $c\in P$.
If $i<j$, every index in the $i$th block precedes every index in the $j$th
block, and hence $a_ia_j=a_ib_j=b_ia_j=b_ib_j=c$. Therefore
\[
    d_id_j=(a_i-b_i)(a_j-b_j)=c-c-c+c=0.
\]
Thus $d_id_j=0$ whenever $i<j$.

If instead the common color is $(\leftarrow,c)$, the same argument gives
$d_jd_i=0$ whenever $i<j$. In either case, the $k+1$ distinct nonzero
elements $d_1,\ldots,d_{k+1}$ form a clique in $\Gamma(R)$, contradicting
$\omega(\Gamma(R))\le k$.
\end{proof}

A ring $P$ is \emph{nilpotent} if $P^t=0$ for some $t\ge1$; the least such
$t$ is called its \emph{nilpotency index}. Much stronger, $P$ is a \emph{zero ring}
if its multiplication is identically zero, i.e., its nipotency index is at most $2$.

\begin{lem}\label{lem:nilpotent-order}
Let $k,t\ge1$, and define $k_0=k$ and
$k_{i+1}=\mathcal R_{2(k_i+1)}\bigl((k_i+1)(k_i+2)\bigr)-1$ for $i\ge0$.
If $P$ is a ring satisfying $P^t=0$ and $\omega(\Gamma(P))\le k$, then
\[
    |P|\le\prod_{i=0}^{t-2}(k_i+1),
\]
where the empty product is $1$.
\end{lem}

\begin{proof}
We argue by induction on $t$. For $t=1$ we have $P=0$, and the result is
immediate. Suppose $t\ge2$, and let
$A=\{a\in P:aP=Pa=0\}$. Then $A$ is a two-sided ideal and a zero subring of $P$, so its nonzero
elements constitute a clique in $\Gamma(P)$ and $|A|\le k+1=k_0+1$. Moreover,
$P^{t-1}\subseteq A$, and hence $(P/A)^{t-1}=0$. By
\Cref{lem:finite-ideal-quotient}, 
\[
    \omega(\Gamma(P/A))
    \le
    \mathcal R_{2|A|}\bigl((k+1)(k+2)\bigr)-1
    \le
    \mathcal R_{2(k+1)}\bigl((k+1)(k+2)\bigr)-1
    =k_1.
\]
Applying the induction hypothesis to $P/A$, with $k_1$ in place of $k$,
gives $|P/A|\le\prod_{i=1}^{t-2}(k_i+1)$. Consequently,
\[
    |P|
    =|A|\,|P/A|
    \le
    (k_0+1)\prod_{i=1}^{t-2}(k_i+1)
    =
    \prod_{i=0}^{t-2}(k_i+1).
\]\end{proof}

We now prove Beck's conjecture for reduced rings. Recall that a ring is said to be \emph{reduced} if it has no nonzero nilpotent elements, and it is \emph{entire} if $xy=0$ implies $x=0$ or $y=0$.
\begin{lem}\label{lem:reduced-ring}
If $B$ is a reduced ring, then
\[
    \chi(\Gamma(B))=\omega(\Gamma(B)).
\]
\end{lem}

\begin{proof}
The theorem of Andrunakievi\v{c}--Rjabuhin
\cite{AndrunakievicRjabuhin68}, in the form proved by Klein \cite{Klein80}, says that every reduced ring $B$ is
a subdirect product of entire rings. More precisely, there is a family
of entire rings $\{D_\lambda\}_{\lambda\in\Lambda}$ and an injective ring
homomorphism
\[
    \varphi:B\hookrightarrow\prod_{\lambda\in\Lambda}D_\lambda
\]
such that every coordinate projection
$\pi_\lambda\circ\varphi:B\to D_\lambda$ is surjective. Thus we may identify
$B$ with a subring of the direct product and write
$x=(x_\lambda)_{\lambda\in\Lambda}\in B$. Since multiplication in the direct
product is coordinatewise and every $D_\lambda$ is entire, if
$\sigma(x)=\{\lambda\in\Lambda:x_\lambda\ne0\}$, then
$\sigma(xy)=\sigma(yx)=\sigma(x)\cap\sigma(y)$.
Hence, $L=\{\sigma(x):x\in B\}$ is a
meet-semilattice under intersection. Moreover, two nonzero elements of $B$
have a zero product in either order exactly when their supports are disjoint. The zero-divisor graph of the meet-semilattice $L$, which has a minimum $\hat{0}$ because of being finite ($\emptyset$ in our case), is defined as the graph on $L$ in which two vertices $x$ and $y$ are adjacent if and only if $x\wedge y=\hat{0}$.
Thus the zero-divisor graph of $B$ is obtained from the zero-divisor graph
of $L$ by replacing each vertex by the nonempty independent set consisting
of the elements with that support. Replacing vertices by nonempty
independent sets preserves both clique and chromatic number, and the coloring
theorem for meet-semilattices \cite[Theorem 1]{NWD07} therefore gives
$\chi(\Gamma(B))=\omega(\Gamma(B))$.
\end{proof}

Here and below, $\oplus$ denotes a
direct sum. Thus, ring operations are componentwise.

\begin{lem}\label{lem:matrix-sum}
Let
\[
    C=\bigoplus_{i=1}^r M_{n_i}(\mathbb F_{q_i}),
    \qquad n_i\ge2,
\]
be a finite direct sum of full matrix rings. If
$\ell=\omega(\Gamma(C))$, then
\[
    r\le\frac{\ell}{3}
    \qquad\text{and}\qquad
    |C|\le\ell^{\frac{5\ell}{3}}.
\]
In particular, $\chi(\Gamma(C))\le\ell^{\frac{5\ell}{3}}$.
\end{lem}

\begin{proof}

For each $i$, choose the three matrices $X_{i,1}, X_{i,2}, X_{i,3}\in M_{n_i}(\F_{q_i})$ satisfying that the only nonzero entry of $X_{i,j}$ equals $1\in \F_{q_i}$ and belongs to the position $(1,1), (2,2)$ and $(1,2)$, respectively.

The elements $X_{1,1},\ldots,X_{r,3}$ are pairwise
annihilating, so they form a clique in $\Gamma(C)$ and hence $\ell\geq 3r$, which implies $r\le\frac{\ell}{3}$. 

Fix a factor $M_n(\mathbb F_q)$. Following the idea of ~\cite[Theorem 8]{AM07}, we choose a decomposition
$\mathbb F_q^n=U\oplus W$ with $\dim U=\lfloor n/2\rfloor$ and
$\dim W=\lceil n/2\rceil$. Relative to this decomposition, the matrices
\[
    Z=
    \left\{
    \begin{pmatrix}
        0&A\\
        0&0
    \end{pmatrix}
    :
    A\in\operatorname{Hom}_{\mathbb F_q}(W,U)
    \right\}
\]
form a zero subring of $M_n(\mathbb F_q)$: every $Y\in Z$ maps
$\mathbb F_q^n$ into $U$, while every $X\in Z$ vanishes on $U$, so $XY=0$
for all $X,Y\in Z$. Moreover,
$Z\cong\operatorname{Hom}_{\mathbb F_q}(W,U)$ as an $\mathbb F_q$-vector
space, and hence
$\dim_{\mathbb F_q}Z=(\dim U)(\dim W)=\lfloor n^2/4\rfloor$. Thus
$Z\setminus\{0\}$ (regarded as a subset of $C$ supported on this factor), together with the matrix whose only nonzero entry equals $1\in\F_q$ and belongs to the first row and column, is
a clique, and therefore
\[
    q^{\lfloor n^2/4\rfloor}\le\ell.
\]

Since $|M_n(\mathbb F_q)|=q^{n^2}$ and
$n^2\le5\lfloor n^2/4\rfloor$ for $n\ge2$, it follows that
\[
    |M_n(\mathbb F_q)|
    =q^{n^2}
    \le q^{5\lfloor n^2/4\rfloor}
    =\left(q^{\lfloor n^2/4\rfloor}\right)^5
    \le\ell^5.
\]

Consequently,
\[
    |C|
    =\prod_{i=1}^r |M_{n_i}(\mathbb F_{q_i})|
    \le\ell^{5r}
    \le\ell^{\frac{5\ell}{3}}.
\]

The final assertion follows from $\chi(\Gamma(C))\le|C|$.
\end{proof}

Before proving \Cref{thm:ring-chi-bounded}, some simple definitions remain.
Given a proper coloring $c$ of the zero-divisor graph $\Gamma(R)$ of a ring $R$, define $\widehat c(x)$ as an extension to all elements of $R$ as follows:
\[
 \widehat c(x)=
 \begin{cases}
   c(x)& \text{ if } x\in V(\Gamma(R)),\\
   \star_0& \text{ if }x=0, \text{ and}\\
   \star_r&\text{ if }x\ne0\text{ and $x$ is not a one-sided zero-divisor},
 \end{cases}
\]
where $\star_0$ is a new color and $\star_r$ is one of the colors in $c$. We have:

\begin{obs}\label{obs:hat}
Let $R$ be a ring and $c$ a proper $\ell$-coloring of $\Gamma(R)$. Then all $x,y\in R$ satisfy
\begin{equation}\label{eq:extended-coloring}
    xy=0\text{ or }yx=0
    \quad\Longrightarrow\quad
    \widehat c(x)\ne\widehat c(y),
\end{equation} 
and $\widehat c$ uses at most
$\ell+1$ colors.
\end{obs}

A proper two-sided ideal $Q\subsetneq R$ is \emph{prime} if for every two-sided ideals
$I,J\subseteq R$, the inclusion $IJ\subseteq Q$ implies $I\subseteq Q$ or
$J\subseteq Q$. The \emph{prime radical} $P=P(R)$ of $R$ is the intersection of all
prime ideals of $R$. An ideal is said to be \emph{nil} when each of its elements is nilpotent, while a ring is said to be \emph{semiprime} when it has no nonzero nilpotent ideals.

\ringchibounded* 
\begin{proof}
If $\Gamma(R)$ has no vertices there is nothing to prove, so let
$k=\omega(\Gamma(R))\ge1$. The nonzero elements of a zero subring constitute a
clique in $\Gamma(R)$, and consequently every zero subring of $R$ has at
most $k+1$ elements.

Let $P=P(R)$ be the prime radical of $R$.
 It is a standard fact that $P$ is nil and $S=R/P$ is semiprime (see~\cite[Section~10]{Lam01}).

Now, Bell proved \cite[Theorem~1]{Bell01} that every infinite nil ring contains an infinite zero subring. Hence, $P$ is finite. Since  every finite nil ring is nilpotent by ~\cite[Corollary~4.13]{Lam01}, $P$ is nilpotent.

First, we bound the nilpotency index $t$ of $P$. For every $\lfloor t/2\rfloor\le i<t$ choose
$x_i\in P^i\setminus P^{i+1}$. These elements are distinct and nonzero. For
distinct $i,j$ in this range we have $i+j\ge t$, and hence
$x_ix_j=x_jx_i=0$. Thus they form a clique of size $\lceil t/2\rceil$ in
$\Gamma(R)$, so $t\le2k$. 
By \Cref{lem:nilpotent-order} applied to $P$, it follows that $|P|$ is bounded by a function of $k$.

Now consider the semiprime ring $S$. A further
application of \Cref{lem:finite-ideal-quotient} shows that
$\ell:=\omega(\Gamma(S))$ is bounded in terms of $k$. This implies that every zero subring of
$S$ is finite, since its nonzero elements form a clique. 

Hence, a result of Bell~\cite[Theorem~2]{Bell01} on 
semiprime rings with finite zero subrings yields
\begin{equation}\label{eq:bell-decomposition}
    S=B\oplus C,
\end{equation}
where $B$ is reduced and $C$ is a direct sum of finitely many full matrix rings over finite fields. 
Without loss of generality, absorb the $1\times1$ matrix factors of $C$ into the reduced summand $B$.

By \Cref{lem:reduced-ring},
$\chi(\Gamma(B))=\omega(\Gamma(B))\le\ell$. Choose an optimal coloring
$c_B$ of $\Gamma(B)$. By~\Cref{obs:hat} the extension $\widehat c_B$ uses at
most $\ell+1$ colors.

After the absorption, $C$ is a finite direct sum of full matrix rings of size at least
$2\times2$. Since $\Gamma(C)$ embeds into $\Gamma(S)$, we have
$\omega(\Gamma(C))\le\ell$, and \Cref{lem:matrix-sum} then gives
\[
    |C|\le \ell^{\frac{5\ell}{3}}.
\]

Color a vertex $(b,c)$ of $\Gamma(S)$ by
$\bigl(\widehat c_B(b),c\bigr)$. If two distinct vertices receive the same
pair, then their $C$-coordinates are equal and their $B$-coordinates are
distinct. If they were adjacent, then their $B$-coordinates would have a
zero product in one of the two orders, contradicting
\eqref{eq:extended-coloring}. Thus we have a proper coloring of $\Gamma(S)$ and
\[
    \chi(\Gamma(S))\le |C|(\ell+1),
\]
which is bounded in terms of $k$.

Choose one representative $s(\overline x)\in R$ for each coset
$\overline x\in S$.  Thus
\[
    s(\overline x)+P=\overline x,
\]
and every $x\in R$ can be written uniquely in the form
\[
    x=s(\overline x)+p_x,
    \qquad\text{with }p_x\in P.
\]
 Choose a proper coloring $c_S$ of
$\Gamma(S)$ with $\chi(\Gamma(S))$ colors and extend it to
$\widehat c_S$. For each vertex $x$ of $\Gamma(R)$, assign the color
$\bigl(\widehat c_S(\overline x),p_x\bigr)$. If distinct vertices $x,y$ receive
the same color, then $p_x=p_y$ and hence $\overline x\ne\overline y$. If, say,
$xy=0$, then $\overline x\,\overline y=0$ in $S$, so \eqref{eq:extended-coloring} gives
$\widehat c_S(\overline x)\ne\widehat c_S(\overline y)$, a contradiction; the case
$yx=0$ is identical. Therefore
\[
    \chi(\Gamma(R))
    \le |P|\bigl(\chi(\Gamma(S))+1\bigr),
\]
and the right-hand side depends only on $k$. This proves the theorem.
\end{proof}

\begin{cor}\label{cor:asymptotic-upper}
The function in \Cref{thm:ring-chi-bounded} can be chosen to be a tower of $2$'s of height $O(k)$, with $k = \omega(\Gamma(R))$.
\end{cor}
\begin{proof}
We use the standard estimate
\[
    \mathcal R_q(s)\le q^{q(s-1)}
\]
(see, for instance, \cite[Chapter~1]{GRS90}).  For the sequence
$(k_i)$ in \Cref{lem:nilpotent-order}, this gives
\[
    k_{i+1}+1
    \le
    \bigl(2(k_i+1)\bigr)^{
        2(k_i+1)((k_i+1)(k_i+2)-1)}
    =
    2^{k_i^{O(1)}}.
\]
Hence each step increases the height of an exponential tower by at most
a constant.

In the proof of \Cref{thm:ring-chi-bounded}, the prime radical $P$ has
nilpotency index at most $2k$.  By \Cref{lem:nilpotent-order}, $|P|$ is
therefore bounded by a tower of $2$'s of height $O(k)$.  Applying
\Cref{lem:finite-ideal-quotient} and the same Ramsey estimate shows that
\[
    \ell:=\omega(\Gamma(R/P))
\]
is bounded by a tower of $2$'s of height $O(k)$ as well.

Finally, \Cref{lem:matrix-sum} gives
$|C|\le\ell^{\frac{5\ell}{3}}$, while the proof of
\Cref{thm:ring-chi-bounded} gives
\[
    \chi(\Gamma(R/P))\le |C|(\ell+1)
    \quad\text{and}\quad
    \chi(\Gamma(R))
    \le |P|\bigl(\chi(\Gamma(R/P))+1\bigr).
\]
These operations increase the tower height by only a constant.  Hence
$f(k)$ can be chosen to be bounded by a tower of $2$'s of height $O(k)$.
\end{proof}

\section{Examples}\label{sec:examples}

In this section, we first construct finite commutative unital rings whose zero-divisor graphs give a superpolynomial lower bound for the binding
function of \Cref{thm:ring-chi-bounded}.  We then further contrast \Cref{thm:ring-chi-bounded} by showing that neither in semirings nor in the nonassociative setting a binding function exists.

\subsection{Several alternating forms}\label{subsec:many-forms}

Let $V,W$ be finite-dimensional vector spaces over $\F_2$, and let
\[
 B:V\times V\rightarrow W
\]
be an \textit{alternating} bilinear map, which means that $B(v,v)=0$ for every $v\in V$.  Since the characteristic of the base field is two, $B$ is
symmetric. A subspace $U\leq V$ is said to be \textit{totally isotropic} for $B$ if $B(U,U)=0$, and we denote by $\iota(B)$ the maximum dimension of a totally isotropic subspace for $B$, i.e. \[
 \iota(B)=\max\{\dim U:U\le V\text{ and } B(U,U)=0\}.
\]

On $A_B=V\oplus W$, we define the associative and commutative multiplication
\begin{equation}\label{eq:two-step-product}
 (u,a)\cdot_{A_B}(v,b)=(0,B(u,v)).
\end{equation}

Notice that for $x=(u,a)$ and $y=(v,b)$ in $A_B$,
\begin{equation}\label{eq:two-step-adjacency}
 xy=0\quad\Longleftrightarrow\quad B(u,v)=0.
\end{equation} 

The following elementary probabilistic lemma is the only existence result
needed for the stronger construction.

\begin{lem}\label{lem:random-isotropic}
For every $t\ge2$ there is an alternating bilinear map
\[
 B:\F_2^n\times\F_2^n\rightarrow\F_2^t,
 \ \text{ with }\ n=\binom{t+1}{2},
\]
such that $\iota(B)\le t$.
\end{lem}
\begin{proof}
Put $V=\F_2^n$, $W=\F_2^t$, and choose uniformly at random a linear map
\[
 \widetilde B:\Lambda^2V\rightarrow W.
\]
It determines an alternating bilinear map $B:V\times V\rightarrow W$ given by $B(v,w)=\widetilde B(v\wedge w)$ for every $v,w\in V$.  Let $r=t+1$.  For any fixed
$r$-dimensional subspace $U\le V$, the restriction of $\widetilde B$ to
$\Lambda^2U$ is uniformly random and $B(U,U)=0$ if and only if $\widetilde B \,\big|_{\Lambda^2U}=0$. Therefore,
\[
 \Pr[B(U,U)=0]=2^{-\dim(W)\dim(\Lambda^2U)}=2^{-t\binom r2}.
\]

 The number of $r$-subspaces of $V$ is the Gaussian binomial coefficient
$\genfrac{[}{]}{0pt}{}{n}{r}_2$, and the product formula gives
\[
 \genfrac{[}{]}{0pt}{}{n}{r}_2=\prod_{i=0}^{r-1}\frac{2^n-2^i}{2^r-2^i}<\prod_{i=0}^{r-1}\frac{2^n}{2^{r-1}}=2^{r(n-r+1)}.
\]

By the union bound,
\[
 \Pr[\text{some $r$-dimensional subspace is totally isotropic}]
 <2^{r(n-r+1)-t\binom r2}
 =2^{-t(t+1)/2}<1.
\]

Thus a suitable $B$ exists.
\end{proof}

\superpolylower*
\begin{proof}
    For every $t\geq 2$, let us consider an alternating bilinear map $B$ satisfying the conditions in the statement of Lemma \ref{lem:random-isotropic} and build the ring $R_t=A_B$ by the construction presented at the beginning of this subsection.
Choose coordinates
\[
 B(u,v)=\bigl(B_1(u,v),\ldots,B_t(u,v)\bigr).
\] 

Let $C$ be a clique in $\Gamma(R_t)$ and let $P\subseteq V$ be the set of projections of its
elements.  By \eqref{eq:two-step-adjacency}, $B(u,v)=0$ for distinct
$u,v\in P$, while $B(u,u)=0$ by alternation.  Hence the subspace
$\Span(P)$ of $V$ generated by the elements of $P$ is totally isotropic, so $\dim(\Span(P))\le t$.  There are at most
$2^t$ elements in $\Span(P)$ and $2^t$ choices of $W$-coordinates for each of
them.  Therefore, excluding 0 we get that
\[
 |C|\le2^{2t}-1.
\]

Let now $I$ be an independent set in $\Gamma(R_t)$.  Its projection to $V$ is injective, since the product of two
distinct elements with the same projection is zero.  Let $S\subseteq
V$ be the set of projections.  For distinct $u,v\in S$ we have
$B(u,v)\ne0$.  For every $u\in S$ we define the indicator over $\mathbb F_2$ of the condition $B(u,x)=0$,
\[
 f_u(x)=\prod_{i=1}^t\bigl(1+B_i(u,x)\bigr),
\]
 and for every $u,v\in S$ we have
\[
 f_u(v)=
 \begin{cases}
  1& \text{if }u=v,\text{ and}\\
  0&\text{if }u\ne v.
 \end{cases}
\]
The functions $f_u$, with $u\in S$, are therefore linearly independent, and each of them
has degree at most $t$ because every $B_i$ is linear.  Every Boolean function, i.e., every function $\F_2^n \rightarrow \F_2$, has a unique
multilinear representative on $\mathbb F_2[x_1,\dots,x_n]$, and a basis of the space of multilinear functions of degree at most
$t$ is the set of all square-free monomials of degree at most $t$ that use at most $n$ variables.  Consequently
\[
 |I|=|S|\le\sum_{j=0}^t\binom nj.
\]

Moreover, the number of vertices of $\Gamma(R_t)$ is essentially $|A_B|=2^{n+t}=2^{t(t+3)/2}=2^{\Theta(t^2)}$ and, with $n=\binom{t+1}{2}$,
\[
 \alpha(\Gamma(R_t))\leq\sum_{j=0}^t\binom nj
 \le (t+1)\binom nt
 \le (t+1)\left(\frac{en}{t}\right)^t
 =2^{O(t\log t)},
\] where we have used the inequality ${n\choose t}\leq \left(\frac{en}{t}\right)^t$ that can be found for example in \cite{AS16}.

This allows us to conclude that $\chi(\Gamma(R_t))\geq \frac{|V(\Gamma(R_t))|}{\alpha(\Gamma(R_t))}\geq \frac{2^{\Theta(t^2)}}{2^{O(t\log t)}}=2^{\Omega(t^2)}$.

Finally, in order to obtain unital rings, we consider the \emph{unitization}  $
 A_B^+=\F_2\oplus A_B$, with multiplication defined as \[(m,r)\cdot_{A_B^+}(n,s)=(m\cdot_{\mathbb F_2}n,\,m\cdot s+n\cdot r+r\cdot_{A_B}s)\]

It is well-known that $A_B^+$ is a finite commutative ring and has $(1,0)$ as its unit element, see~\cite[p.~11]{Schafer66}.

Since $A_B$ is isomorphic to $\{0\}\oplus A_B$, we can see $A_B$ as a subset of $A_B^+$. Under this consideration, we can say that every element outside $A_B$ is its own inverse and, hence, a unit in $A_B^+$, while $x^2=0$ for every $x\in A_B$.

Hence the set of  zero-divisors of $A_B^+$ is
\[
 V(\Gamma(A_B^+))=A_B\setminus\{0\},
\]
and given $(0,r),(0,s)\in A_B^+$, then $(0,r)\cdot_{A_B^+}(0,s)=(0,r\cdot_{A_B}s)$, so the edges of $\Gamma(A_B^+)$ are determined by (\ref{eq:two-step-adjacency}).

Consequently, the zero-divisor graphs of $A_B^+$ and $A_B$ are the same and we get an infinite family of finite commutative unital rings $R_t$ with \[ 
\omega(\Gamma(R_t))\leq 2^{O(t)}\qquad \text{and} \qquad \chi(\Gamma(R_t))\geq 2^{\Omega(t^2)}.\] 
\end{proof}

\subsection{Semirings}\label{subsec:semirings}

\semiring*
\begin{proof}
By \cite[Theorem~3]{DD05}, there exists a finite commutative
$3$-nilpotent semigroup $S$ with zero $0$ such that $
\Gamma(S)\cong G$. One such semigroup is  $S = \{0,c,s_1,\ldots,s_n\}$, where the element $c\in S$ corresponds to the dominating
vertex of $G$, and $s_1,\dots,s_n$ correspond to the remaining vertices. The operation in $S$ is defined as follows: $cs=sc=0$ for every $s\in S$, $s_i^2=c$ for every $i\in\{1,\dots,n\}$, and for $i\neq j$ \[
s_is_j=s_js_i=\begin{cases}
    0 & \text{ if }\{s_i,s_j\}\in E(G), \text{ and}\\
    c & \text{ if }\{s_i,s_j\}\notin E(G).
\end{cases}
\]

Adjoin an identity $1$ to $S$, obtaining the commutative monoid $S^1$, and set $R=\mathcal P(S^1\setminus\{0\})$ to be the set of subsets of $S^1\setminus\{0\}$. 
For $A,B\in R$, define
$$
A+B=A\cup B
\qquad\text{and}\qquad
AB=\{ab:a\in A,\ b\in B,\ ab\neq0\}.
$$

Then $R$ is a finite commutative unital semiring with idempotent addition.
Every nonempty subset of $S\setminus{0}$ is a zero-divisor, since it is
annihilated by $\{c\}$, whereas no subset containing $1$ is a zero-divisor.
Hence, $
V(\Gamma(R))
=
\mathcal P(S\setminus\{0\})\setminus\{\emptyset\}
$ and the vertices $
X=\bigl\{\{s\}:s\in S\setminus\{0\}\bigr\}
$ induce a copy of $\Gamma(S)\cong G$.

Fix a linear order on $S\setminus\{0,c\}$ and define $
r:V(\Gamma(R))\rightarrow X
$ by

$$
r(A)=
\begin{cases}
\{c\}&\text{if }A=\{c\},\text{ and}\\
\{\min(A\setminus\{c\})\}&\text{if }A\neq\{c\}.
\end{cases}
$$

If $A, B$ are adjacent in $\Gamma(R)$ and neither is $\{c\}$, then, writing
$r(A)={x}$ and $r(B)={y}$, we have $xy=0$.  Moreover, $x\neq y$ since
in the above realization $x^2\neq0$.  Hence ${x}$ and ${y}$ are adjacent
in $X$.  If one of $A,B$ equals $\{c\}$, their images are adjacent because
$c$ corresponds to the dominating vertex of $G$.  Thus $r$ is a graph
homomorphism.  Since it fixes $X$ pointwise, it is a retraction onto
$X\cong G$.
\end{proof}

\subsection{An explicit nonassociative construction}\label{subsec:free-nonassoc}
A commutative
nonassociative $K$-algebra is a $K$-vector space with a commutative bilinear
multiplication; see, for example, \cite{Schafer66}.  We use the same
definition of $\Gamma(R)$ as before, since it only involves the binary multiplication.

Given a field $K$ and symbols $V={x_1, \ldots, x_n}$ called \emph{generators}, the \emph{free commutative nonassociative unital $K$-algebra} $A^K_V$ can be seen as the set of all formal linear $K$-combinations of rooted binary trees whose leaves are labeled by $V$, together with the empty tree $1$, see~\cite{Petrogradsky05}. Multiplication of two nonempty binary trees is just hanging them under a new common root, while $1$ is the identity.

Given a finite graph $G=(V,E)$ without isolated vertices, define
\begin{equation}\label{eq:free-quotient}
A_G=A^K_V
\Big/\bigl\langle x_ux_v:uv\in E\bigr\rangle.
\end{equation}

Hence, a nonempty tree survives in $A_G$ if and only if no internal vertex has two children that are leaves labeled $u$ and
$v$ for some edge $uv$ of $G$.

\free*
\begin{proof}
First observe that no element with nonzero scalar part is a zero-divisor. Indeed, suppose
$$
(\lambda 1+a)b=0
$$
with $\lambda\neq0$, $a$ a linear combination of nonempty trees, and $b\neq0$. The scalar part of $b$ must be zero. Choose a tree of smallest size occurring in $b$. It occurs with nonzero coefficient in $\lambda b$, while every tree occurring in $ab$ is strictly larger, since multiplication by a nonempty tree strictly increases the size. Hence this term cannot cancel, a contradiction.

Suppose now $a,b\in \Gamma(A_G)$ and $ab=0$, written as linear combinations of nonempty trees. Observe that no tree $s$ can appear in both of $a,b$, since then $s^2$ would occur in $ab$ with nonzero coefficient.

Now let trees $s$ and $t$ belong to $a$ and $b$, respectively. The coefficient of the tree $st$ in $ab$ cannot cancel with any other term. Hence $st$ must be zero in $A_G$. Since $s$ and $t$ themselves survive in $A_G$, the forbidden product must occur at the new root. Hence both are generators, say $s=x_u$ and $t=x_v$, with $uv\in E$. Consequently every zero-divisor of $A_G$ is a linear combination of generators, and if $a=\sum_{u\in U}\alpha_u x_u$ and $b=\sum_{v\in W}\beta_v x_v$, then $ab=0$ if and only if $uv\in E$ for all $u\in U$ and $v\in W$.

Since $G$ has no isolated vertices, every generator $x_v$ is a zero-divisor, and $x_ux_v=0$ if and only if $uv\in E$. Thus the generators induce a copy of $G$ in $\Gamma(A_G)$. Conversely, choosing one vertex from the support of each zero-divisor defines a graph homomorphism $\Gamma(A_G)\to G$ which fixes the generators, and hence $G$ is a retract of $\Gamma(A_G)$.
\end{proof}

\subsection{Finite commutative nonassociative examples}
\label{subsec:finite-nonassoc}
We use the classical theorem of Erd\H{o}s that for all $g,k$ there is a finite
graph of girth at least $g$ and chromatic number at least $k$
\cite{Erdos59}.  We also use the elementary fact that if $Q,V$ are
$\F_2$-vector spaces with $\dim V=n$ and a linear map $T:Q\to V$ is chosen
uniformly at random, then for every $d$-dimensional subspace $L\le Q$,
\begin{equation}\label{eq:random-linear}
 \Pr[T(L)=0]=2^{-nd}.
\end{equation}

\nonassocunbounded*
\begin{proof}

For every $k$ we can choose a finite graph $G$ on $[n]$ with 
\[
 \girth(G)\ge6\text{ and } \chi(G)\ge k.
\]

Let $V=\F_2^n$ with basis
$e_1,\ldots,e_n$, and consider
\[
 F_G=\Span\{e_i\wedge e_j\,:\,\{i,j\}\in E(G)\}\le\Lambda^2V
 \text{ and } Q=\Lambda^2V/F_G.
\]

Let $\pi:\Lambda^2V\to Q$ be the quotient map.  Choose a random linear map
$T:Q\to V$ and define a multiplication on $V$ by
\begin{equation}\label{eq:nonassoc-product}
 xy:=T\bigl(\pi(x\wedge y)\bigr).
\end{equation}
This multiplication is commutative since $x\wedge y=y\wedge x$ in characteristic two, and $x^2=0$ for all $x\in V$.  Thus, we can consider $R:=V$ as a commutative ring with the previous multiplication, and every nonzero element of $R$ is a
zero-divisor.

If $\{i,j\}\in E(G)$, then $e_ie_j=0$.  Hence $G$ is a subgraph of $\Gamma(R)$ and $\chi(\Gamma(R))\geq k$.

Let $U\le V$ have dimension five and choose a row-echelon basis of $U$ with
pivot coordinates $p_1,\ldots,p_5$.  The ten wedge products of these basis
vectors have distinct leading coordinates the coefficients of
$e_{p_i}\wedge e_{p_j}$.  Consequently
\[
 \dim(\Lambda^2U\cap F_G)
 \le |E(G[\{p_1,\ldots,p_5\}])|.
\]
Since $G$ has girth at least six, the latter graph is a forest and has at most
four edges.  Hence
\[
 \dim\pi(\Lambda^2U)\ge\binom52-4=6.
\]

For this fixed $U$, using (\ref{eq:random-linear}) we have that
\[
 \Pr[U^2=0]=\Pr[T(\pi(\Lambda^2U))=0]\le2^{-6n}.
\]

There are fewer than $2^{5n}$ five-dimensional subspaces of $V$, so 
\[
 \Pr[\text{some five-space }U\text{ satisfies }U^2=0]<2^{-n}<1.
\]

Fix $T$ for which this bad event does not occur.

We know that $\chi(\Gamma(R))\ge k$.  If $C$ is a clique in $\Gamma(R)$,
then $xy=0$ for all $x,y\in C$, including $x=y$.  By bilinearity,
$\Span(C)^2=0$.  Thus $\dim\Span(C)\le4$ and
\[
 |C|\le2^4-1=15.
\] 

As we did in the proof of ~\Cref{thm:superpoly-lower}, to obtain commutative nonassociative rings which are also unital and satisy the conditions in the statement, we consider the unitization $R^+=\F_2\oplus R$, where \[(m,r)\cdot_{R^+}(n,s)=(m\cdot_{\mathbb F_2}n,\,m\cdot s+n\cdot r+T(\pi(r\wedge s))).\]

We have that $\chi(\Gamma(R^+))\geq\chi(\Gamma(R))\geq k$ and all the new vertices in $\Gamma(R^+)$ constitute an independent set, so the clique number can be increased in at most one unit and we get $\omega(\Gamma(R^+))\leq 16$.
\end{proof}

\subsection*{Statement on the use of AI} Generative AI tools contributed to the exploration and development of initial ideas for the proofs. The authors subsequently undertook substantial work to develop these ideas, verify and improve the arguments, and write, refine, and revise the manuscript as a whole. The authors take full responsibility for the correctness of the results and the final content of the paper.

\subsection*{Acknowledgments}
The authors were partially supported by the grant PID2022-137283NB-C22 funded by MICIU/AEI/10.13039/501100011033 and by ERDF/EU. Sara Asensio was also supported by European Social Fund Plus, Programa Operativo de Castilla y León and Junta de Castilla y León via its Consejería de Educación (and also via the project with reference CLU-2025-1-02-IMUVA). Kolja Knauer was also supported through the Severo Ochoa and Mar\'ia de Maeztu Program for Centers and Units of Excellence in R\&D (CEX2020-001084-M).

\small
\bibliography{Coloring_rings_lit_v12}

@book{AS16,
author = {Alon, Noga and Spencer, Joel H.},
title = {The Probabilistic Method},
year = {2016},
isbn = {1119061954},
publisher = {Wiley Publishing},
edition = {4th}
}

@article{AndrunakievicRjabuhin68,
  author  = {Andrunakievi{\v c}, V. A. and Rjabuhin, Ju. M.},
  title   = {Rings without nilpotent elements and completely simple ideals},
  journal = {Soviet Mathematics Doklady},
  volume  = {9},
  year    = {1968},
  pages   = {565--567}
}

@article{DS98,
 author = {Dumaldar, Mahesh N. and Sharma, Pramod K.},
 title = {Comments over ``{Some} non-chromatic rings''},
 fjournal = {Communications in Algebra},
 journal = {Communications in Algebra},
 issn = {0092-7872},
 volume = {26},
 number = {11},
 pages = {3871--3883},
 year = {1998},
 language = {English},
 doi = {10.1080/00927879808826380},
 zbMATH = {1218351},
 Zbl = {0908.05045}
}

@article{Kav22,
 author = {Kavaskar, Thanthony},
 title = {Beck's coloring of finite product of commutative ring with unity},
 fjournal = {Graphs and Combinatorics},
 journal = {Graphs and Combinatorics},
 issn = {0911-0119},
 volume = {38},
 number = {2},
 pages = {},
 note = {34},
 year = {2022},
 language = {English},
 doi = {10.1007/s00373-021-02401-x},
 zbMATH = {7490307},
 Zbl = {1482.13012}
}

@Article{PWJ17,
 Author = {Patil, Avinash and Waphare, B. N. and Joshi, Vinayak},
 Title = {Perfect zero-divisor graphs},
 FJournal = {Discrete Mathematics},
 Journal = {Discrete Mathematics},
 ISSN = {0012-365X},
 Volume = {340},
 Number = {4},
 Pages = {740--745},
 Year = {2017},
 Language = {English},
 DOI = {10.1016/j.disc.2016.11.027},
 zbMATH = {6680918},
 Zbl = {1355.05115}
}

@Article{DMS02,
 Author = {DeMeyer, F. R. and McKenzie, T. and Schneider, K.},
 Title = {The zero-divisor graph of a commutative semigroup},
 FJournal = {Semigroup Forum},
 Journal = {Semigroup Forum},
 ISSN = {0037-1912},
 Volume = {65},
 Number = {2},
 Pages = {206--214},
 Year = {2002},
 Language = {English},
 DOI = {10.1007/s002330010128},
 zbMATH = {1837299},
 Zbl = {1011.20056}
}

@InCollection{AB17,
 Author = {Anderson, David F. and Badawi, Ayman},
 Title = {The zero-divisor graph of a commutative semigroup: a survey},
 BookTitle = {Groups, modules, and model theory -- surveys and recent developments. In memory of R\"udiger G\"obel. Proceedings of the conference on new pathways between group theory and model theory, M\"ulheim an der Ruhr, Germany, February 1--4, 2016},
 ISBN = {978-3-319-51717-9; 978-3-319-51718-6},
 Pages = {23--39},
 Year = {2017},
 Publisher = {Cham: Springer},
 Language = {English},
 DOI = {10.1007/978-3-319-51718-6_2},
 zbMATH = {7220024},
 Zbl = {1436.20113}
}

@Article{DD05,
 Author = {DeMeyer, Frank and DeMeyer, Lisa},
 Title = {Zero divisor graphs of semigroups.},
 FJournal = {Journal of Algebra},
 Journal = {Journal of Algebra},
 ISSN = {0021-8693},
 Volume = {283},
 Number = {1},
 Pages = {190--198},
 Year = {2005},
 Language = {English},
 DOI = {10.1016/j.jalgebra.2004.08.028},
 zbMATH = {2153359},
 Zbl = {1077.20069}
}

@article{Beck88,
 author = {Beck, Istv{\'a}n},
 title = {Coloring of commutative rings},
 journal = {Journal of Algebra},
 volume = {116},
 number = {1},
 pages = {208--226},
 year = {1988},
 doi = {10.1016/0021-8693(88)90202-5}
}

@article{AN93,
 author = {Anderson, D. D. and Naseer, M.},
 title = {Beck's coloring of a commutative ring},
 journal = {Journal of Algebra},
 volume = {159},
 number = {2},
 pages = {500--514},
 year = {1993},
 doi = {10.1006/jabr.1993.1171}
}

@article{BDS98,
 author = {Bhatwadekar, S. M. and Dumaldar, Mahesh N. and Sharma, Pramod K.},
 title = {Some non-chromatic rings},
 journal = {Communications in Algebra},
 volume = {26},
 number = {2},
 pages = {477--505},
 year = {1998},
 doi = {10.1080/00927879808826143}
}

@article{Vietri13,
 author = {Vietri, Andrea},
 title = {A combinatorial analysis of zero-divisor graphs on certain polynomial rings},
 journal = {Communications in Algebra},
 volume = {41},
 number = {6},
 pages = {2040--2047},
 year = {2013},
 doi = {10.1080/00927872.2011.653068}
}

@article{Vietri15,
 author = {Vietri, Andrea},
 title = {A new zero-divisor graph contradicting {Beck's} conjecture, and the classification for a family of polynomial quotients},
 journal = {Graphs and Combinatorics},
 volume = {31},
 number = {6},
 pages = {2413--2423},
 year = {2015},
 doi = {10.1007/s00373-014-1501-6}
}

@article{TLW20,
 author = {Tang, Gaohua and Lin, Guangke and Wu, Yansheng},
 title = {Associate class graph of zero-divisors of a commutative ring},
 journal = {Journal of Algebra and Its Applications},
 volume = {19},
 number = {8},
 pages={},
 note = {2050155},
 year = {2020},
 doi = {10.1142/S0219498820501558}
}

@incollection{AAS11,
 author = {Anderson, David F. and Axtell, Michael C. and Stickles, Joe A. Jr.},
 title = {Zero-divisor graphs in commutative rings},
 booktitle = {Commutative Algebra: Noetherian and Non-Noetherian Perspectives},
 editor = {Fontana, Marco and Kabbaj, Salah-Eddine and Olberding, Bruce and Swanson, Irena},
 publisher = {Springer},
 address = {New York},
 pages = {23--45},
 year = {2011},
 doi = {10.1007/978-1-4419-6990-3_2}
}

@article{ScottSeymour20,
author  = {Scott, Alex and Seymour, Paul},
title   = {A survey of $\chi$-boundedness},
journal = {Journal of Graph Theory},
volume  = {95},
number  = {3},
pages   = {473--504},
year    = {2020},
doi     = {10.1002/jgt.22601}
}

@article{Reiher26,
author  = {Reiher, Christian},
title   = {Graphs of large girth},
journal = {Computer Science Review},
volume  = {62},
pages   = {},
note={100858},
year    = {2026},
doi     = {10.1016/j.cosrev.2025.100858}
}

@article{HP24,
 author = {Hala\v{s}, Radom{\'i}r and P{\'o}cs, Jozef},
 title = {On zero-divisor graphs of infinite posets},
 journal = {Soft Computing},
 volume = {28},
 pages = {12113--12118},
 year = {2024},
 doi = {10.1007/s00500-024-09958-8}
}

@article{DolzanOblak12,
author = {Dol{\v z}an, David and Oblak, Polona},
title = {The zero-divisor graphs of rings and semirings},
journal = {International Journal of Algebra and Computation},
volume = {22},
number = {4},
pages = {1250033},
year = {2012},
doi = {10.1142/S0218196712500336}
}

@article{AL99,
 author = {Anderson, David F. and Livingston, Philip S.},
 title = {The zero-divisor graph of a commutative ring},
 journal = {Journal of Algebra},
 volume = {217},
 number = {2},
 pages = {434--447},
 year = {1999},
 doi = {10.1006/jabr.1998.7840}
}

@article{NWD07,
 author = {Nimbhorkar, S. K. and Wasadikar, M. P. and DeMeyer, Lisa},
 title = {Coloring of meet-semilattices},
 journal = {Ars Combinatoria},
 volume = {84},
 pages = {97--104},
 year = {2007}
}

@book{GRS90,
  author    = {Graham, Ronald L. and Rothschild, Bruce L. and Spencer, Joel H.},
  title     = {Ramsey Theory},
  edition   = {2},
  publisher = {John Wiley \& Sons},
  address   = {New York},
  year      = {1990}
}

@article{Redmond02,
  author  = {Shane P. Redmond},
  title   = {The zero-divisor graph of a non-commutative ring},
  journal = {International Journal of Commutative Rings},
  volume  = {1},
  number  = {4},
  pages   = {203--211},
  year    = {2002}
}

@article{AM06,
  author  = {Saieed Akbari and Ali Mohammadian},
  title   = {Zero-divisor graphs of non-commutative rings},
  journal = {Journal of Algebra},
  volume  = {296},
  number  = {2},
  pages   = {462--479},
  year    = {2006},
  doi     = {10.1016/j.jalgebra.2005.07.007}
}

@article{AM07,
  author  = {Saieed Akbari and Ali Mohammadian},
  title   = {On zero-divisor graphs of finite rings},
  journal = {Journal of Algebra},
  volume  = {314},
  number  = {1},
  pages   = {168--184},
  year    = {2007},
  doi     = {10.1016/j.jalgebra.2007.02.051}
}

@article{Wu05,
  author  = {Tongsuo Wu},
  title   = {On directed zero-divisor graphs of finite rings},
  journal = {Discrete Mathematics},
  volume  = {296},
  number  = {1},
  pages   = {73--86},
  year    = {2005},
  doi     = {10.1016/j.disc.2005.03.006}
}

@book{Lam01,
  author    = {T. Y. Lam},
  title     = {A First Course in Noncommutative Rings},
  edition   = {2},
  series    = {Graduate Texts in Mathematics},
  volume    = {131},
  publisher = {Springer},
  address   = {New York},
  year      = {2001}
}

@article{TW06,
  author  = {Tang, Zhongming and Wan, Zhe-Xian},
  title   = {Symplectic graphs and their automorphisms},
  journal = {European Journal of Combinatorics},
  volume  = {27},
  number  = {1},
  year    = {2006},
  pages   = {38--50},
  doi     = {10.1016/j.ejc.2004.08.002}
}

@article{GodsilRoyle01,
  author  = {Godsil, Chris D. and Royle, Gordon F.},
  title   = {Chromatic Number and the 2-Rank of a Graph},
  journal = {Journal of Combinatorial Theory, Series B},
  volume  = {81},
  number  = {1},
  pages   = {142--149},
  year    = {2001},
  doi     = {10.1006/jctb.2000.2003}
}

@article{Bell01,
  author  = {Howard E. Bell},
  title   = {On zero subrings and periodic subrings},
  journal = {International Journal of Mathematics and Mathematical Sciences},
  volume  = {28},
  number  = {7},
  year    = {2001},
  pages   = {413--417},
  doi     = {10.1155/S0161171201006044}
}

@article{Klein80,
  author  = {Abraham A. Klein},
  title   = {A Simple Proof of a Theorem on Reduced Rings},
  journal = {Canadian Mathematical Bulletin},
  volume  = {23},
  number  = {4},
  year    = {1980},
  pages   = {495--496},
  doi     = {10.4153/CMB-1980-075-7}
}

@article{Erdos59,
  author  = {Erd{\H{o}}s, Paul},
  title   = {Graph theory and probability},
  journal = {Canadian Journal of Mathematics},
  volume  = {11},
  year    = {1959},
  pages   = {34--38},
  doi     = {10.4153/CJM-1959-003-9}
}

@book{Schafer66,
  author    = {Schafer, Richard D.},
  title     = {An Introduction to Nonassociative Algebras},
  series    = {Pure and Applied Mathematics},
  volume    = {22},
  publisher = {Academic Press},
  address   = {New York},
  year      = {1966}
}

@article{Petrogradsky05,
  author  = {Petrogradsky, Victor M.},
  title   = {Enumeration of algebras close to absolutely free algebras
             and binary trees},
  journal = {Journal of Algebra},
  volume  = {290},
  year    = {2005},
  pages   = {337--371},
  doi     = {10.1016/j.jalgebra.2004.01.034}
}
\bibliographystyle{my-siam}

\end{document}